\documentclass[12pt, reqno]{amsart}
\usepackage{amsmath, amsthm, amscd, amsfonts, amssymb, mathrsfs, graphicx, color}
\usepackage[bookmarksnumbered, colorlinks, plainpages]{hyperref}
\hypersetup{colorlinks=true,linkcolor=red, anchorcolor=green, citecolor=cyan, urlcolor=red, filecolor=magenta, pdftoolbar=true}

\newtheorem{theorem}{Theorem}[section]
\newtheorem{lemma}[theorem]{Lemma}
\newtheorem{proposition}[theorem]{Proposition}
\newtheorem{corollary}[theorem]{Corollary}
\theoremstyle{definition}

\theoremstyle{remark}
\newtheorem{remark}[theorem]{Remark}
\numberwithin{equation}{section}

\begin{document}
\title[Norm inequalities related to the Heron and Heinz means]{Norm inequalities related to the Heron and Heinz means}

\author[Y. Kapil, C. Conde, M.S. Moslehian, M. Singh, M. Sababheh]{Yogesh Kapil$^1$, Cristian Conde$^2$, Mohammad Sal Moslehian$^3$, Mandeep Singh$^1$ and Mohammad Sababheh$^4$}

\address{$^1$ Department of Mathematics, Sant Longowal Institute of Engineering and
Technology, Longowal-148106, Punjab, India}
\email{yogesh\_kapill@yahoo.com} \email{msrawla@yahoo.com}

\address{$^2$Instituto de Ciencias, Universidad Nacional de Gral. Sarmiento, J.
M. Gutierrez 1150, (B1613GSX) Los Polvorines and Instituto Argentino de Matem\'atica ``Alberto P. Calder\'on", Saavedra 15 3º piso, (C1083ACA) Buenos Aires, Argentina}
\email{cconde@ungs.edu.ar}

\address{$^3$Department of Pure Mathematics, Center of Excellence in
Analysis on Algebraic Structures (CEAAS), Ferdowsi University of
Mashhad, P. O. Box 1159, Mashhad 91775, Iran}
\email{moslehian@um.ac.ir, moslehian@member.ams.org}

\address{$^4$ Department of mathematics, University of Sharjah, Sharjah, UAE.\newline Department of basic Sciences, Princess Sumaya University for Technology, Amman 11941, Jordan} \email{sababheh@psut.edu.jo, sababheh@yahoo.com}


\maketitle

\begin{abstract}
In this article, we present several inequalities treating operator means and the Cauchy-Schwarz inequality. In particular, we present some new comparisons between operator Heron and Heinz means, several generalizations of the difference version of the Heinz means and further refinements of the Cauchy-Schwarz inequality. The techniques used to accomplish these results include convexity and L\"{o}wner matrices.
\end{abstract}

\section{Introduction}

There are different families of means that interpolate between the arithmetic  and  geometric means. For example, the Heron and Heinz means,  defined respectively by
\begin{eqnarray*}
F_{\nu}(a,b)=(1-\nu)\sqrt{ab}+\nu \frac{a+b}{2} \qquad {\rm and} \qquad H_{\nu}(a,b)= \frac{a^{1-\nu}b^{\nu}+a^{\nu}b^{1-\nu}}{2},
\end{eqnarray*}
for $a,b\geq 0$ and $\nu\in[0,1]$. It is easy to see that $F$ is an increasing function and $H$ is a symmetric and convex function in $\nu$ on $[0,1].$ Hence,
\begin{eqnarray}\label{HH}
\sqrt{ab}\leq F_{\nu}(a,b)\leq \frac{a+b}{2}\qquad {\rm and} \qquad \sqrt{ab}\leq H_{\nu}(a,b)\leq \frac{a+b}{2}.
\end{eqnarray}

Recall that the arithmetic--geometric mean inequality $\sqrt{ab}\leq \frac{a+b}{2}$ can be expressed by using the Heron and Heinz means as follows:
\begin{eqnarray*}
F_0(a,b)=H_{1/2}(a,b)\leq F_1(a,b).
\end{eqnarray*}

In \cite{Bh}, Bhatia compared these families of means by showing that
\begin{eqnarray}\label{bhatia}
H_{\nu}(a,b)\leq F_{(2\nu-1)^2}(a,b),
\end{eqnarray}
for all $\nu\in [0,1]$. One goal of this article is to present a new comparison between $H_{\nu}$ and $F_{\nu}$, by means of the Kantorovich constant. More precisely, we will show that $$H_{\nu}(a,b)\leq \left(\frac{a+b}{2\sqrt{ab}}\right)^{1-\nu}F_{\nu}(a,b), 0\leq \nu\leq 1.$$
In the sequel, we set some basic preliminary backgrounds that will be needed throughout the paper.

Let $\mathbb{B}(\mathscr{H})$ denote the $C^*$-algebra of all bounded linear operators acting on a separable complex Hilbert space
$( \mathscr{H},\langle \cdot,\cdot\rangle).$ The cone of positive operators is denoted by $\mathbb{B}(\mathscr{H})_{+}$. Let $\mathbb{K}(\mathscr{H})$ denote the ideal of compact operators in $\mathbb{B}(\mathscr{H})$. For any compact operator $A\in \mathbb{K}(\mathscr{H})$,
let $s_{1}(A), s_{2}(A),\cdots$ be the eigenvalues of $|A|= (A^*A)^{1\over{2}}$ arranged in decreasing order and repeated according to
multiplicity. If $A \in \mathcal{M}_n$ (the algebra of all $n\times n$ matrices over $\mathbb{C}$), we take
$s_k(A)=0$ for $k>n$. We denote by $\mathbb{B}(\mathscr{H})_{+}$ (resp., $\mathcal{M}_n^{+}$) the cone of positive operators (resp., positive definite matrices), while $\mathbb{B}(\mathscr{H})_{++}$ (resp., $\mathcal{M}_n^{++}$) stands for the set of invertible operators in $\mathbb{B}({\mathscr{H}})_{+}$. A unitarily invariant norm in $\mathbb{K}(\mathscr{H})$ is a map
$|||\cdot|||: \mathbb{K}(\mathscr{H}) \to [0,\infty]$
given by $|||A|||=g(s(A))$, $A \in \mathbb{K}(\mathscr{H})$, where $g$ is a symmetric gauge function; cf. \cite{GK}.
The set $\mathcal{I}=\{A \in \mathbb{K}(\mathscr{H}):|||A||| < \infty \}$
is a (two-sided) ideal of $\mathbb{B}(\mathscr{H})$. The operator norm $\|\cdot\|$ and the Schatten $p$-norms
$\|A\|_p=\left(\sum_{j} s_j^p(A)\right)^{1/p}$ for $p \geq 1$ are significant examples of the unitarily invariant norms. For notational convenience, we shall denote $(\mathcal{I}, |||.|||)$ by $\mathcal{I}$.

The inequalities in \eqref{HH} have some possible operator versions as follows.\\
If $A, B \in \mathbb{B}(\mathscr{H})_+,$ $X\in \mathcal{I}$ and $\nu \in [0,1]$, then
\begin{eqnarray*}
2|||A^{1/2}XB^{1/2}|||\leq |||A^{\nu}XB^{1-\nu}+A^{1-v}XB^{\nu}|||\leq |||AX+XB|||.
\end{eqnarray*}

Recently, Kapil and Singh \cite[Theorems 3.7 and 3.8]{KS1} proved that if $A, B \in \mathbb{B}(\mathscr{H})_+,$ $X\in \mathcal{I}$ then
\begin{eqnarray}\label{kaursingh}
\frac 12 |||A^{\nu}XB^{1-\nu}+A^{1-v}XB^{\nu}|||\leq \left|\left|\left| (1-\alpha)A^{\frac 12}XB^{\frac 12}+\alpha \left(\frac{AX+XB}{2}\right)\right|\right|\right|,
\end{eqnarray}
for $1/4\leq \nu \leq 3/4$ and $1/2\leq \alpha <\infty$. A comparison between the geometric and Heron means is a particular case of \eqref{kaursingh}, when $\nu=1/2$, i.e.,
\begin{eqnarray}\label{kaursingh2}
 |||A^{1/2}XB^{1/2}|||\leq \left|\left|\left| (1-\alpha)A^{\frac 12}XB^{\frac 12}+\alpha \left(\frac{AX+XB}{2}\right)\right|\right|\right|.
\end{eqnarray}
 Further in \cite{KS1}, authors proved some generalizations of the difference version of Heinz inequality, given by
 \begin{eqnarray}\label{ks2}
 |||A^{\nu}XB^{1-\nu}-A^{1-\nu}XB^{\nu}|||\leq |2\nu-1|\;|||AX-XB|||,\;\;\;\;\;\hspace{1cm} (0\leq \nu\leq 1).
\end{eqnarray}
 Moreover, it is proved that $|||A^{\nu}XB^{1-\nu}-A^{1-\nu}XB^{\nu}|||$ is a convex function of $\nu$; see \cite[Remark 3.12]{KS1}. These results have also been proved in matrix version by several authors; see \cite{KS,Zh} for example.

The aim of this paper is to obtain refinements of inequalities (\ref{kaursingh}) and (\ref{kaursingh2}). Some refinements in difference version of Heinz inequality are also obtained with some generalizations. Then we utilize the upper and lower bounds for the normalized Jensen functional (see Theorem \ref{Mi}) on the convexity of several functions observed in this study. This leads to more refinements of norm inequalities. At the end, we utilize the Jensen functional once more to discuss some refinements of the Cauchy--Schwarz inequality. We refer the reader to \cite{AFA} for recent developments of the Cauchy--Schwarz inequality.

\section{Background}

Throughout this note, we denote by $J$ a closed interval of the real line and $f$ is assumed to be a continuous real-valued function defined on $J$.
In 1906, J. Jensen introduced the concept of {\it Jensen convex} (or {\it midpoint convex}) {\it function}, characterized by
\begin{eqnarray}\label{j}
f\left(\frac{x+y}{2}\right) \leq \frac{f(x)+f(y)}{2},\
\end{eqnarray}
for all $x,y\in J.$ That is, these are functions that behave in a particular way under the action of the arithmetic mean. Note that the arithmetic--geometric mean inequality is a particular case of (\ref{j}) by considering $f(x)=e^x$. In the context of continuity, midpoint convexity gives rise to {\it convexity}. That is,
\begin{eqnarray*}f(\lambda x+ (1-\lambda)y) \leq \lambda f(x)+(1-\lambda)f(y),
\end{eqnarray*}
for all $x,y\in J$ and $0\leq\lambda\leq1.$

It is well known that every convex function on a closed interval can be modified at the
endpoints to become convex and continuous. An immediate consequence of this
remark is the integrability of $f$. The integral of $f$ can then be estimated by
\begin{eqnarray}\label{s}
f\left(\frac{x+y}{2}\right) \leq \frac{1}{y-x}\int_x^yf(t)\,dt \leq \frac{f(x)+f(y)}{2}\,.
\end{eqnarray}

This fundamental inequality, which was first published by Hermite in 1883 and independently proved in 1893 by Hadamard, is well known as the Hermite--Hadamard inequality. It is obvious that \eqref{s} is an interpolating inequality for \eqref{j}.

There is a growing literature considering several interesting generalizations, refinements and interpolations in various frameworks.
We would like to refer the reader to \cite{Con, Dr} and references therein for more information. Some mathematicians have obtained several refinements of the operator inequalities as consequences of the Hermite--Hadamard inequality, for example \cite{KMSC, K, MOS}.

If $f$ is a convex function on $J$, then the well-known Jensen's inequality asserts that
\begin{eqnarray*}
0\leq \sum_1^n p_i f(x_i)-f\left(\sum_1^n p_i x_i\right):=\mathcal{J}(f,{\bold x}, {\bold p}),
\end{eqnarray*}
where ${\bold x}=(x_1,\cdots, x_n)\in J^n$ and ${\bold p}=(p_1,\cdots,p_n), p_i\geq 0$ with $\sum p_i=1.$ $\mathcal{J}$ is called the normalized Jensen functional and in recent years, many authors have studied it and have established upper and lower bounds for this functional; see  \cite{Dr} as an example.

\begin{theorem} \cite[Corollary 1]{Dr}\label{Mi}
Let $f$ be a convex function on $J$. Then
 \begin{align}\label{uljensen}
 2\lambda_{min} &\left(\frac{f(x_1)+f(x_2)}{2}-f\left(\frac{x_1+x_2}{2}\right)\right)\leq \lambda f(x_1)+(1-\lambda)f(x_2)-
f(\lambda x_1+(1-\lambda)x_2)
 \nonumber \\ &\leq
 2\lambda_{max}\left(\frac{f(x_1)+f(x_2)}{2}-f\left(\frac{x_1+x_2}{2}\right)\right), \
 \end{align}
 where $0\leq \lambda \leq 1$, $\lambda_{min}=\min\{\lambda,1-\lambda\}$, $\lambda_{max}=\max\{\lambda,1-\lambda\}$ and $x_1, x_2\in J.$
 \end{theorem}

First note that by integrating \eqref{uljensen} over $[0,1]$ we obtain
\begin{eqnarray}\label{integral}
\frac 12\left(\frac{f(x_1)+f(x_2)}{2}-f\left(\frac{x_1+x_2}{2}\right)\right)&\leq &\frac 12 f(x_1)+ \frac 12f(x_2)- \frac{1}{x_2-x_1}\int_{x_1}^{x_2}f(x) dx \nonumber \\ &\leq&
 \frac 32 \left(\frac{f(x_1)+f(x_2)}{2}-f\left(\frac{x_1+x_2}{2}\right)\right).
\end{eqnarray}

That is, we have upper and lower bounds for the difference between the terms that appear in the left of \eqref{s}.

\section{Norm inequalities involving operator version of Heron and Heinz means}

For the sake of simplicity, we denote

\begin{eqnarray*}
F(\nu)=\frac {1}{2} |||A^{\nu}XB^{1-\nu}+A^{1-v}XB^{\nu}|||, \;\;\;\;\;\; K(\nu)=|||A^{\nu}XB^{1-\nu}-A^{1-v}XB^{\nu}|||,
\end{eqnarray*}

and
$$
G(\nu)= \left|\left|\left| (1-\nu)A^{\frac 12}XB^{\frac 12}+\nu \left(\frac{AX+XB}{2}\right)\right|\right|\right|\,,
$$
for $A, B \in \mathbb{B}(\mathscr{H})_+,$ $X\in \mathcal{I}$ and $\nu \in[0,1].$
We remind the reader of a result in \cite[Remarks 3.2 and 3.12]{KS1} that the functions $F(\nu)$ and $K(\nu)$ are convex on $[0, 1]$ and attain their minimum at $\nu=1/2.$

As we have mentioned at the introduction, the authors of \cite{KS1} have obtained a complete interpolation and comparison of operator inequalities for Heron and Heinz means. More precisely, they proved that $G(\nu)\leq G(1/2)$ for $\nu\in [0,1/2], G(\nu)$ is an increasing function for $\nu \in [1/2, \infty)$ and
\begin{eqnarray}\label{ksreducida}
F(\nu)\leq G(\alpha),
\end{eqnarray}
for $\nu \in [1/4,3/4]$ and $\alpha \geq 1/2.$ \\

\begin{theorem}\label{t1} Let $1/4\leq \nu\leq 3/4$ and $\alpha \in [1/2, \infty)$. Then
\begin{eqnarray}\label{ays1}
F(\nu)\leq (4r_0-1)F(1/2)+ 2(1-2r_0)G(\alpha)\leq G(\alpha),
\end{eqnarray}
and
\begin{eqnarray}\label{ays1b}
F(1/2)+2\left(2\int_{1/4}^{3/4} F(\nu) d\nu - F(1/2)\right)\leq G(\alpha).
\end{eqnarray}
where $r_0(\nu)=\min\{\nu, 1-\nu\}.$
\end{theorem}
\begin{proof}
We first choose $\frac{1}{4} \leq \nu \leq \frac{1}{2}.$ Then using the convexity of $F(\nu)$ (see \cite[Remark 3.2]{KS1}), we obtain
\begin{eqnarray}\label{e1}
F(\nu)=F\left((2-4\nu)\frac{1}{4}+(4\nu-1)\frac{1}{2}\right) \leq (2-4\nu)F(1/4)+(4\nu-1)F(1/2).
\end{eqnarray}
Now using \eqref{ksreducida} with $\nu = 1/4$ in \eqref{e1}, we get
\begin{eqnarray}\label{e2}
 F(\nu)\leq (4\nu-1)F(1/2)+ (2-4\nu)G(\alpha),
\end{eqnarray}
which is equivalent to
\begin{eqnarray*}
F(\nu)\leq (4r_0-1)F(1/2)+ 2(1-2r_0)G(\alpha)\leq G(\alpha),
\end{eqnarray*}
 for $\frac{1}{4} \leq \nu \leq \frac{1}{2}.$\\
If $\frac{1}{2} \leq \nu \leq \frac{3}{4},$ replace $\nu$ by $1-\nu$ in \eqref{e2} to get
\begin{eqnarray}\label{e3} F(\nu)&\leq& (3-4\nu)F(1/2)+ 2(2\nu-1)G(\alpha)\nonumber \\
&=&(4r_0-1)F(1/2)+ 2(1-2r_0)G(\alpha)\nonumber \\
&\leq& G(\alpha),\end{eqnarray}
as $1-\nu=r_0$ in this case. This completes the proof of the first conclusion.
The second conclusion follows by taking the sum of integrals with respect to $\nu$ of \eqref{e2} and \eqref{e3} over $[1/4,1/2]$ and $[1/2,3/4]$, respectively.
\end{proof}

A matrix version of Theorem \ref{t1} has been proved by Ali et al. in \cite[Theorem 2.2 and 2.5]{AYS}.

\begin{theorem}\label{t20}
Let $1/4\leq \nu\leq 3/4$ and $\alpha \in [1/2, \infty)$. Then
\begin{eqnarray*}
(4r_0-1)F(1/2)+ 2(1-2r_0)G(\alpha)\leq 2r_{2} F(1/2)+(1-2r_{2})G(\alpha),
\end{eqnarray*}
where $r_{2}(\nu)=\min\{2\nu-\frac 12, |1-2\nu|, \frac32-2\nu\}$ and $r_0(\nu)=\min\{\nu, 1-\nu\}.$
\end{theorem}

\begin{proof}
Let $l_1=(4r_0-1)F(1/2)+ 2(1-2r_0)G(\alpha)$ and $l_2=2r_{2} F(1/2)+(1-2r_{2})G(\alpha).$ By a simple calculation, we have
\[ l_1-l_2 = \left\{ \begin{array}{cccc}
 0 & \mbox{if $\nu \in [1/4, 3/8]\cup [5/8, 3/4]$};\\
 (8\nu-3)F(1/2)+ (3-8\nu)G(\alpha) & \mbox{if $\nu \in [3/8, 1/2]$}; \\
 (-8\nu+5)F(1/2)+ (8\nu-5)G(\alpha) & \mbox{if $\nu \in [1/2, 5/8].$}\end{array} \right. \]

So, by the inequality $F(1/2)\leq G(\alpha)$, we  conclude that $l_1-l_2\leq 0.$
\end{proof}

\begin{remark} On combining the results of Theorems \ref{t1} and \ref{t20}, we obtain the following double inequality,
\begin{eqnarray*}
\frac{1}{2}|||A^{\nu}X\hspace{-01cm}&&B^{1-\nu}+A^{1-\nu}XB^{\nu}|||\nonumber\\
&\leq &(4r_{0}-1)|||A^{1/2}XB^{1/2}|||+2(1-2r_{0})|||(1-\alpha)A^{1/2}XB^{1/2}+\alpha\left(\frac{AX+XB}{2}\right)|||\nonumber\\
&\leq & 2r_{2}|||A^{1/2}XB^{1/2}|||+(1-2r_{2})|||(1-\alpha)A^{1/2}XB^{1/2}+\alpha\left(\frac{AX+XB}{2}\right)|||
\end{eqnarray*}
for $A, B \in \mathbb{B}(\mathscr{H})_+,$ $X\in \mathcal{I}.$
This not only refines an inequality proved by Kaur et al. in \cite{KMSC} but also lifts that from a matrix version to an operator one.
\end{remark}

Our next result is a new comparison between the Heron and Heinz means. First, a scalar version will be given.
\begin{proposition}
Let $a,b>0$ and let $0\leq \nu\leq 1.$ Then
\begin{equation}\label{heron_heinz_comp}
H_{\nu}(a,b)\leq \left(\frac{H_1(a,b)}{H_{\frac{1}{2}}(a,b)}\right)^{1-\nu}F_{\nu}(a,b).
\end{equation}
\end{proposition}
\begin{proof}
Without loss of generality, we may assume $a=1.$ Then the desired inequality reduces to
\begin{equation}\label{wanted_heron_heinz}
\frac{b^{\nu}+b^{1-\nu}}{2}\leq\left(\frac{1+b}{2\sqrt{b}}\right)^{1-\nu}\left((1-\nu)\sqrt{b}+\nu\frac{1+b}{2}\right).
\end{equation}
To prove this inequality, let
$$f(\nu)=\log(b^{\nu}+b^{1-\nu})-(1-\nu)\log\frac{1+b}{2\sqrt{b}}-\log\left(2(1-\nu)\sqrt{b}+\nu(1+b)\right).$$
Calculus computations show that
$$f''(\nu)=\frac{(-1 + \sqrt{b})^4}{(-2 \sqrt{b} (-1 + \nu) + \nu + b \nu)^2 }+ \frac{ 4 b^{1 + 2 \nu} \log^2 b}{(b + b^{2 \nu})^2}.$$ It is clear that $f''(\nu)\geq 0$ for $0\leq \nu\leq 1.$ Hence, $f$ is convex on $[0,1]$. But then $f(\nu)\leq \max\{f(0),f(1)\}=0.$ Since $f(\nu)\leq 0,$ we have
$$\log(b^{\nu}+b^{1-\nu})\leq(1-\nu)\log\frac{1+b}{2\sqrt{b}}+\log\left(2(1-\nu)\sqrt{b}+\nu(1+b)\right),$$ which is equivalent to \eqref{wanted_heron_heinz}.
\end{proof}
Notice that \eqref{heron_heinz_comp} reads as
$$\frac{a^{\nu}b^{1-\nu}+a^{1-\nu}b^{\nu}}{2}\leq\left(\frac{a+b}{2\sqrt{ab}}\right)^{1-\nu}\left((1-\nu)\sqrt{ab}+\nu\frac{a+b}{2}\right).$$ The factor $\frac{a+b}{2\sqrt{ab}}$ has appeared in recent studies of means refinements. The quantity $\left(\frac{a+b}{2\sqrt{ab}}\right)^2$ has been referred to as the Kantorovich constant. We refer the reader to \cite{Liao_Wu} and its references as a sample of some work treating this constant.

Our next result is a matrix version of \eqref{heron_heinz_comp}.
\begin{corollary}
Let $A,B\in\mathcal{M}_n^{+}, X\in\mathcal{M}_n$ and $0\leq \nu\leq 1.$ If there are two positive numbers $m,M$ such that $m\leq A,B\leq M,$ then
$$\left\|\frac{A^{\nu}XB^{1-\nu}+A^{1-\nu}XB^{\nu}}{2}\right\|_2\leq \left(\frac{m+M}{2\sqrt{mM}}\right)^{1-\nu}\left\|(1-\nu)A^{\frac{1}{2}}XB^{\frac{1}{2}}+\nu\frac{AX+XB}{2}\right\|_2.$$
\end{corollary}
\begin{proof}
Let $A=U{\text{diag}}(\lambda_i)U^*$ and $B=V{\text{diag}}(\mu_j)V^*$ be the spectral decompositions of $A$ and $B$, respectively. Letting $U^*XV=Y,$ we have
\begin{align*}
\frac{A^{\nu}XB^{1-\nu}+A^{1-\nu}XB^{\nu}}{2}&=U\frac{{\text{diag}}(\lambda_i^{\nu})Y{\text{diag}}(\mu_j^{1-\nu})+
{\text{diag}}(\lambda_i^{1-\nu})Y{\text{diag}}(\mu_j^{\nu})}{2}V^*\\
&=U\frac{[\lambda_i^{\nu}\mu_j^{1-\nu}+\lambda_i^{1-\nu}\mu_j^{\nu}]\circ [y_{ij}]}{2}V^*,
\end{align*}
where $\circ$ stands for the Schur product. Since $\|\cdot\|_2$ is unitarily invariant and recalling \eqref{heron_heinz_comp}, we get
\begin{align*}
\left\|\frac{A^{\nu}XB^{1-\nu}+A^{1-\nu}XB^{\nu}}{2}\right\|_2^2&=\sum_{i,j}
\left(\frac{\lambda_i^{\nu}\mu_j^{1-\nu}+\lambda_i^{1-\nu}\mu_j^{\nu}}{2}\right)^2|y_{ij}|^2\\
&\leq \sum_{i,j}\left(\frac{\lambda_i+\mu_j}{2\sqrt{\lambda_i\mu_j}}\right)^{2(1-\nu)}
\left((1-\nu)\lambda_i^{\frac{1}{2}}\mu_j^{\frac{1}{2}}+\nu\frac{\lambda_i+\mu_j}{2}\right)^2|y_{ij}|^2\\
&\leq \left(\frac{m+M}{2\sqrt{mM}}\right)^{2(1-\nu)}\left\|(1-\nu)A^{\frac{1}{2}}XB^{\frac{1}{2}}+\nu\frac{AX+XB}{2}\right\|_2^2,
\end{align*}
where we have used the fact that $m\leq \lambda_i,\mu_j\leq M$ to obtain the last line.
\end{proof}
\section{The difference version of Heinz inequality}
In this section, we still adopt the predefined function $K(\nu)=|||A^{\nu}XB^{1-\nu}-A^{1-\nu}XB^{\nu}|||.$
\begin{theorem}\label{t2}Let $1/4\leq \nu\leq 3/4$. Then
\begin{eqnarray}\label{ays3}
K(\nu)\leq 2(1-2r_0)K(1/4),
\end{eqnarray}
and
\begin{eqnarray}\label{ays3b}
\int\limits_{1/4}^{3/4}K(\nu) d\nu \leq\frac{1}{4}K(1/4) \leq\frac{1}{8} K(1),
\end{eqnarray}
where $r_0(\nu)=\min\{\nu, 1-\nu\}.$
\end{theorem}
\begin{proof}We first prove the result for $1/4\leq \nu \leq 1/2.$ By a simple calculation, we obtain $r_0(\nu) \in [1/4,1/2]$ and
$\nu=\frac{2(1-2r_{0})}{4}+\frac{4r_{0}-1}{2}.$
Now, using the convexity of $K(\nu),$ \cite[Remark 3.12]{KS1}, we obtain
$$K(\nu)\leq 2(1-2r_0)K(1/4)+(4r_{0}-1)K(1/2)=2(1-2r_0)K(1/4),$$
as $K(1/2)=0.$
The case $1/2\leq \nu \leq 3/4$ follows from the symmetry of function $K(\nu)$ about the line $\nu = 1/2.$\\
The second result is due to the sum of integrals of \eqref{ays3} with respect to $\nu$ over $[1/4,1/2]$ and [1/2, 3/4], respectively.
\end{proof}

\begin{remark}We remark that inequality \eqref{ays3} in Theorem \ref{t2} can be written as,
$$ |||A^{\nu}XB^{1-\nu}-A^{1-\nu}XB^{\nu}||| \leq 2(1-2r_0)|||A^{1/4}XB^{3/4}-A^{3/4}XB^{1/4}|||,$$
equivalently,
\begin{eqnarray}\label{eq5} |||A^{\nu}XB^{1-\nu}-A^{1-\nu}XB^{\nu}||| \leq 2|1-2\nu|\;|||A^{1/4}XB^{3/4}-A^{3/4}XB^{1/4}|||.\end{eqnarray}
Now recall \eqref{ks2} for $\nu=1/4$ or $3/4,$ we obtain,
\begin{eqnarray}\label{eq6} |||A^{1/4}XB^{3/4}-A^{3/4}XB^{1/4}|||\leq \frac{1}{2}|||AX-XB|||.\end{eqnarray}
On combining \eqref{eq5} and \eqref{eq6}, we obtain
\begin{eqnarray*} |||A^{\nu}XB^{1-\nu}-A^{1-\nu}XB^{\nu}||| &\leq& 2|1-2\nu|\;|||A^{1/4}XB^{3/4}-A^{3/4}XB^{1/4}|||\nonumber \\
&\leq& |1-2\nu||||AX-XB|||.
\end{eqnarray*}
 This proves that \eqref{ays3} in Theorem \ref{t2} interpolates \eqref{ks2}. Similarly \eqref{ays3b} refines the integral version of \eqref{ks2}.
\end{remark}
Before stating the next generalization of the difference version of Heinz inequality, we remind two lemmas. For the first lemma, we refer the reader to \cite[p. 343]{horn}. For the used notation, $Y\circ Z$ refers to the Schur (Hadamard) product of $Y$ and $Z$. That is, it is the entrywise multiplication of $Y$ and $Z$.
\begin{lemma}\label{normofshur}
If $Y\in\mathcal{M}_n^+$ and $Z\in\mathcal{M}_n$ then
$$|||Y\circ Z|||\leq \max_{i}y_{ii}\,|||Z|||.$$
\end{lemma}
A good reference for the following Lemma is \cite{bhatiapositive}.
\begin{lemma}\label{loweneroft^r}
If $(\mu_i)$ are positive numbers, then for $0\leq r\leq 1,$ the matrix $Y$ whose entries are
$$y_{ij}=\left\{\begin{array}{cc}\frac{\mu_i^r-\mu_j^r}{\mu_i-\mu_j},&\mu_i\not=\mu_j\\r \mu_i^{r-1},&\mu_i=\mu_j
\end{array}\right.$$
is positive definite.
\end{lemma}
Now we have the following generalization of the difference version of Heinz inequality.
\begin{theorem}\label{first_gener_neg_heinz}
Let $A,B\in\mathcal{M}_n^{++}$ and $X\in\mathcal{M}_n$. Then for $ {\alpha}\geq 1, \frac{1-{\alpha}}{2}\leq\nu\leq\frac{1+{\alpha}}{2},$ and any unitarily invariant norm $|||\cdot|||$ on $\mathcal{M}_n$,
\begin{eqnarray*}
&&{\alpha}|||A^{\nu}XB^{1-\nu}-A^{1-\nu}XB^{\nu}|||\\
&\leq&|2\nu-1|\max(\|A^{1-{\alpha}}\|,\|B^{1-{\alpha}}\|)\;\;|||A^{{\alpha}}X-XB^{{\alpha}}|||,
\end{eqnarray*}
where $\|\cdot\|$ is the operator norm.
\end{theorem}
\begin{proof}
It suffices to prove the required inequality for the special case when $A=B$ and $A$ is diagonal. Then the general case follows by replacing $A$ with
$\left(\begin{array}{cc}A&0\\0&B\end{array}\right)$ and $X$ with $\left(\begin{array}{cc}0&X\\0&0\end{array}\right)$.\\
So, assume $A={\text{diag}}(\lambda_i)>0,$ and let $W=A^{\nu}XA^{1-\nu}-A^{1-\nu}XA^{\nu}.$ Then $W=Y\circ Z$ where
$Z=A^{\alpha}X-XA^{\alpha}$ and $$Y=\left\{\begin{array}{cc}\frac{\lambda_i^{\nu}\lambda_j^{1-\nu}-
\lambda_{i}^{1-\nu}\lambda_j^{\nu}}{\lambda_i^{\alpha}-\lambda_j^{\alpha}}.&\lambda_i\not=\lambda_j\\ \frac{2\nu-1}{{\alpha}\lambda_i^{{\alpha}-1}},&\lambda_i=\lambda_j\end{array}\right..$$ Observe that when  $\frac{1}{2}\leq\nu\leq\frac{1+{\alpha}}{2},$ we have $0\leq \frac{2\nu-1}{{\alpha}}\leq 1$ and hence, $Y\geq 0$ because
$$
y_{ij}=\lambda_i^{1-\nu}\left(\frac{(\lambda_i^{\alpha})^{\frac{2\nu-1}{{\alpha}}}-(\lambda_j^{\alpha})
^{\frac{2\nu-1}{{\alpha}}}}
{\lambda_i^{\alpha}-\lambda_j^{\alpha}}\right)\lambda_j^{1-\nu}$$ when $ \lambda_i\not=\lambda_j$ and $
y_{ii}=\lambda_i^{1-\nu}\left(\frac{2\nu-1}{{\alpha}}\lambda_i^{2\nu-1-{\alpha}}\right)\lambda_i^{1-\nu}$
by virtue of Lemma \ref{loweneroft^r}, on letting $r=\frac{2\nu-1}{{\alpha}}$ and $\mu_i=\lambda_i^{\alpha}$. Consequently, by Lemma \ref{normofshur},
\begin{eqnarray*}
|||W|||&\leq&\max_{i}y_{ii}\;|||Z|||\\
&=&\frac{1}{{\alpha}}(2\nu-1)\|A^{1-{\alpha}}\|\;|||Z|||.
\end{eqnarray*}
Now, if $\frac{1-{\alpha}}{2}\leq\nu\leq\frac{1}{2}$, we have $Y\leq 0$, hence $W=|Y|\circ (-Z),$ which then implies the result for these values of $\nu.$
\end{proof}

Another generalization of the difference version reads as follows.
\begin{theorem}
Let $A,B\in \mathcal{M}_n^{++}$, $X\in\mathcal{M}_n$ and $0<r\leq 1$. Then
$$|||A^rX-XB^r|||\leq r\max(\|A^{r-1}\|,\|B^{r-1}\|)|||AX-XA|||.$$
\end{theorem}
\begin{proof}
This follows immediately by noting that $A^rX-XA^r=Y\circ Z$ where $A={\text{diag}}(\lambda_i),$ $Z=AX-XA$ and
$$y_{ij}=\left\{\begin{array}{cc}\frac{\lambda_i^r-\lambda_j^r}{\lambda_i-\lambda_j},&\lambda_i\not=\lambda_j\\ r\lambda_i^{r-1},&
\lambda_i=\lambda_j\end{array}\right..$$ Then arguing like Theorem \ref{first_gener_neg_heinz} implies the required inequality.
\end{proof}

On the other hand, a reverse of the difference version of the Heinz inequality maybe obtained as follows.

\begin{proposition}
Let $A, B \in \mathbb{B}(\mathscr{H})_{++},$ $X\in \mathcal{I}$ and let $\nu\not\in [0,1].$ Then
$$|||A^{1-\nu}XB^{\nu}-A^{\nu}XB^{1-\nu}|||\geq |2\nu-1|\;|||AX-XB|||.$$
\end{proposition}
\begin{proof}
For $C,D\in \mathbb{B}(\mathscr{H})_{+}, Z\in \mathcal{I}$ and $0\leq \mu\leq 1,$ we have
\begin{equation}\label{needed_rever_nega_heinz}
|||C^{\mu}ZD^{1-\mu}-C^{1-\mu}ZD^{\mu}|||\leq |2\mu-1|\;|||CZ-ZD|||.
\end{equation}
Now if $\nu\not\in [0,1]$, let $\mu=\frac{\nu}{2\nu-1}.$ Then $\mu\in [0,1].$ For $A, B \in \mathbb{B}(\mathscr{H})_{++}, X\in\mathcal{I},$ and let
$$C=A^{2\nu-1}, Z=A^{1-\nu}XB^{1-\nu}\;{\text{and}}\;D=B^{2\nu-1}.$$ Then substituting these parameters in \eqref{needed_rever_nega_heinz} implies the desired inequality.
\end{proof}

As mentioned in the introduction, in \cite[Remark 3.12]{KS1} it is proved that the function $\nu\mapsto |||A^{1-\nu}XB^{\nu}-A^{\nu}XB^{1-\nu}|||$ is convex on $[0,1].$ In the next result, we extend this convexity to $\mathbb{R}$. The proof of this result is based on some delicate manipulations of the given parameters. The computations follow the same reasoning as in the proof of  \cite[Theorem 4, p. 14]{SabMIA}, and hence, we do not include them here.

\begin{proposition}
Let $A, B \in \mathbb{B}(\mathscr{H})_{++},$ $X\in \mathcal{I}$ and let $K(\nu)=|||A^{1-\nu}XB^{\nu}-A^{\nu}XB^{1-\nu}|||.$ Then $f$ is convex on $\mathbb{R}$.
\end{proposition}

This convexity entails the following difference version of the   Heinz inequality. The proof follows immediately from \cite[Theorem 1, p.4 and Theorem 2, p.6]{SabMIA}, taking $a=0, b=1.$
\begin{corollary}
Let $A, B \in \mathbb{B}(\mathscr{H})_{++},$ $X\in \mathcal{I}, \nu\geq 0$ and let $N\in\mathbb{N}$. Then
$$K(0)+\sum_{j=1}^{N}2^j\nu\left[\frac{K(0)+K(2^{1-j})}{2}-K(2^{-j})\right]\leq K(-\nu).$$ On the other hand, if $\nu\leq -1$, then
$$K(0)-\sum_{j=1}^{N}2^j(1+\nu)\left[\frac{K(1)+K(1-2^{1-j})}{2}-K(1-2^{-j})\right]\leq K(-\nu).$$
\end{corollary}

For example, when $N=1$, the first inequality of the above corollary reduces to
$$(1+2\nu) |||AX-XB|||\leq |||A^{1+\nu}XB^{-\nu}-A^{-\nu}XB^{1+\nu}|||.$$

\section{Consequences of Jensen functionals of convex functions for norm inequalities}
Now, we are in a situation to obtain the following results which are the refinements of \eqref{ays1} and \eqref{ays1b}.

\begin{theorem}\label{conde}
Let $1/4\leq \nu\leq 3/4$ and $\alpha \in [1/2, \infty)$. Then
\begin{eqnarray}\label{refks}
F(\nu)&\leq& F(\nu)+2 \lambda_{\min}\left(\frac{F(1/4)+F(1/2)}{2}-F\left(\frac{1/4+1/2}{2}\right)\right) \nonumber \\&\leq&(4r_0-1)F(1/2)+ 2(1-2r_0)G(\alpha)\leq G(\alpha),
\end{eqnarray}
where $r_0=\min\{\nu, 1-\nu\}$ and $\lambda_{\min}=\min\{2-4r_0, 4r_0-1\}$.
\end{theorem}
\begin{proof}
First, we consider the case $\nu \in [1/4,1/2]$. Then we choose $\lambda \in [0,1]$ as $\lambda =2-4\nu $, i.e., $\nu=\frac{\lambda}{4} +\frac{1-\lambda}{2}.$ Using \eqref{uljensen} we obtain that
\begin{eqnarray*}
F(\nu)&+&2 \lambda_{\min}\left(\frac{F(1/4)+F(1/2)}{2}-F\left(\frac{3/4}{2}\right)\right)\nonumber \\&\leq&2(1-2\nu)F(1/4)+(4\nu-1)F(1/2),
\end{eqnarray*}
where $\lambda_{\min}=\min\{\lambda, 1-\lambda\}=\min\{2-4\nu, 4\nu-1\}$.
By \eqref{ksreducida}, we have
\begin{eqnarray*}
F(\nu)&+&2 \lambda_{\min}\left(\frac{F(1/4)+F(1/2)}{2}-F\left(\frac{3/4}{2}\right)\right)\nonumber \\&\leq&(4\nu-1)F(1/2)+ 2(1-2\nu)G(\alpha).
\end{eqnarray*}
So,
\begin{eqnarray*}
F(\nu)&+&2 \lambda_{\min}\left(\frac{F(1/4)+F(1/2)}{2}-F\left(\frac{3/4}{2}\right)\right)\nonumber \\&\leq&(4r_0-1)F(1/2)+ 2(1-2r_0)G(\alpha).
\end{eqnarray*}
Similarly, for $\nu \in [1/2,3/4]$ we have
\begin{eqnarray*}
F(\nu)&+&2 \lambda_{\min}\left(\frac{F(3/4)+F(1/2)}{2}-F\left(\frac{5/4}{2}\right)\right)\nonumber \\&\leq&2(2\nu-1)F(3/4)+(3-4\nu)F(1/2),
\end{eqnarray*}
where $\lambda_{\min}=\min\{\lambda, 1-\lambda\}=\min\{4\nu-2, 3-4\nu\}$. Using again \eqref{ksreducida}, we get
\begin{eqnarray*}
F(\nu)&+&2 \lambda_{\min}\left(\frac{F(3/4)+F(1/2)}{2}-F\left(\frac{5/4}{2}\right)\right)\nonumber \\&\leq&(4r_0-1)F(1/2)+ 2(1-2r_0)G(\alpha).
\end{eqnarray*}
As $F$ is symmetric about $\nu=1/2,$ we get the desired result.
\end{proof}

\begin{theorem}\label{integral}Let $1/4\leq \nu\leq 3/4$ and $\alpha \in [1/2, \infty)$. Then
\begin{eqnarray*}
F(1/2)&+&\left(\frac{F(1/4)+F(1/2)}{2}-F\left(\frac{1/4+1/2}{2}\right) \right)+2\left(2\int_{1/4}^{3/4} F(\nu) d\nu - F(1/2)\right) \nonumber \\&\leq&G(\alpha).
\end{eqnarray*}
\end{theorem}
\begin{proof}

First, we consider the case $\nu\in [1/4,1/2]$. Integrating inequality \eqref{refks} we obtain that
\begin{eqnarray*}
\int_{1/4}^{1/2} F(\nu)d\nu &+& 2\left(\frac{F(1/4)+F(1/2)}{2}-F\left(\frac{1/4+1/2}{2}\right)\right)\int_{1/4}^{1/2} \lambda_{\min}(\nu) d\nu \nonumber \\
&\leq& F(1/2) \int_{1/4}^{1/2} 4\nu-1 d\nu + G(\alpha)\int_{1/4}^{1/2} 2(1- 2\nu) d\nu\nonumber\\
&\leq& 1/4 G(\alpha),
\end{eqnarray*}
or equivalently,
\begin{eqnarray}\label{int1}
\int_{1/4}^{1/2} F(\nu)d\nu &+& \frac 18 \left(\frac{F(1/4)+F(1/2)}{2}-F\left(\frac38\right)\right)\nonumber \\
&\leq& \frac 18 F(1/2) + \frac 18 G(\alpha)\leq \frac14 G(\alpha).
\end{eqnarray}
Mimicking the same idea in the interval $[1/2,3/4]$ and using the symmetry of $F(\nu),$ we get
\begin{eqnarray}\label{int2}
\int_{1/2}^{3/4} F(\nu)d\nu &+& \frac 18 \left(\frac{F(1/4)+F(1/2)}{2}-F\left(\frac 38\right)\right)\nonumber \\
&\leq& \frac 18 F(1/2) + \frac 18 G(\alpha)\leq \frac14 G(\alpha).
\end{eqnarray}
Adding inequalities \eqref{int1} and \eqref{int2}, we have
\begin{eqnarray*}
\int_{1/4}^{3/4} F(\nu)d\nu &+& \frac 14 \left(\frac{F(1/4)+F(1/2)}{2}-F\left(\frac38\right)\right)\nonumber \\
&\leq& \frac 14 F(1/2) + \frac 14 G(\alpha)\leq \frac12 G(\alpha).
\end{eqnarray*}
Finally, we conclude that
\begin{eqnarray*}
F(1/2)+2\left(2 \int_{1/4}^{3/4} F(\nu)d\nu - F(1/2)\right)+\left(\frac{F(1/4)+F(1/2)}{2}-F\left(\frac38\right)\right) \leq G(\alpha).\nonumber \
\end{eqnarray*}
\end{proof}

Next, we prove the results refining \eqref{ays3} and \eqref{ays3b}.
\begin{theorem}\label{t3}
Let $1/4\leq \nu\leq 3/4$. Then
\begin{eqnarray}\label{y1}
K(\nu)\leq K(\nu)+ 2\lambda_{\min}\left(\frac{1}{2}K(1/4)-K(3/8)\right)\leq 2(1-2r_0)K(1/4),
\end{eqnarray}
where $r_0=\min\{\nu, 1-\nu\}$ and $\lambda_{\min}=\min\{2-4r_0, 4r_0-1\}$.
\end{theorem}
\begin{proof}
First, we consider the case $\nu \in [1/4,1/2]$. Then we choose $\lambda \in [0,1]$ as $\lambda =2-4\nu $, i.e., $\nu=\frac{\lambda}{4}+\frac{1-\lambda}{2}.$ Using \eqref{uljensen} we obtain that
\begin{eqnarray*}
K(\nu)&+&2 \lambda_{\min}\left(\frac{K(1/4)+K(1/2)}{2}-K\left(\frac{3/4}{2}\right)\right)\nonumber \\&\leq&2(1-2\nu)K(1/4)+(4\nu-1)K(1/2),
\end{eqnarray*}
where $\lambda_{\min}=\min\{\lambda, 1-\lambda\}=\min\{2-4\nu, 4\nu-1\}$.
Since $K(1/2)=0,$ we have
\begin{eqnarray*}
K(\nu)+2 \lambda_{\min}\left(\frac{1}{2}K(1/4)-K(3/8)\right)\leq 2(1-2\nu)K(1/4).
\end{eqnarray*}
So,
\begin{eqnarray*}
K(\nu)+2 \lambda_{\min}\left(\frac{1}{2}K(1/4)-K(3/8)\right)\leq 2(1-2r_0)K(1/4).
\end{eqnarray*}
Similarly, for $\nu \in [1/2,3/4],$ we have
\begin{eqnarray*}
K(\nu)&+&2 \lambda_{\min}\left(\frac{K(3/4)+K(1/2)}{2}-K\left(\frac{5/4}{2}\right)\right)\nonumber \\&\leq&2(2\nu-1)K(3/4)+(3-4\nu)K(1/2),
\end{eqnarray*}
where $\lambda_{\min}=\min\{\lambda, 1-\lambda\}=\min\{4\nu-2, 3-4\nu\}$. Since $K(1/2)=0,$ we get
\begin{eqnarray*}
K(\nu)+2 \lambda_{\min}\left(\frac{1}{2}K(3/4)-K(5/8)\right)\leq 2(1-2r_0)K(3/4).
\end{eqnarray*}
As $K$ is symmetric respect to $\nu=1/2,$ we get the desired result.
\end{proof}

\begin{theorem}\label{t4}
The following inequality holds,
\begin{eqnarray}\label{y2}
\int\limits_{1/4}^{3/4} K(\nu) d\nu \leq \frac{1}{8}K(1/4)+\frac{1}{4}K(3/8).
\end{eqnarray}
\end{theorem}
\begin{proof}
By simple calculations we obtain
$$\lambda_{min} = \left\{\begin{array}{c}
4\nu-1\;\;\;\;\;\; for \;\;\;\; 1/4 \leq \nu \leq 3/8 \\
2-4\nu\;\;\;\;\;\; for \;\;\;\; 3/8 \leq \nu \leq 1/2 \\
4\nu-2\;\;\;\;\;\; for \;\;\;\; 1/2 \leq \nu \leq 5/8 \\
3-4\nu\;\;\;\;\;\; for \;\;\;\; 5/8 \leq \nu \leq 3/4 .
\end{array}
\right. $$
Now, taking the sum of integrals of \eqref{y1} with $\lambda_{min}$ as above in the respective intervals and suitable $r_{0},$ keeping in view symmetry of $K(\nu)$ about the line $\nu=\frac{1}{2},$ we obtain the required result.
\end{proof}

\begin{remark}
We claim that inequality \eqref{y2} interpolates \eqref{ays3b}. Indeed,
$$ \int\limits_{1/4}^{3/4} K(\nu) d\nu \leq \frac{1}{8}K(1/4)+\frac{1}{4}K(3/8) \leq \frac{1}{4}K(1/4),$$
noting that $K(3/8)\leq 1/2K(1/4)+1/2K(1/2)$, which follows from convexity of $K(\nu).$
\end{remark}

\begin{remark}
Recently, Bhatia proved in \cite{Bh2} the following inequality in matrix version for the case of the Schatten $2$-norm
\begin{eqnarray*}
\frac12\|A^{\nu}B^{1-\nu}+B^{\nu}A^{1-\nu}\|_2\leq \frac 12 \|A^{\nu}B^{1-\nu}+A^{1-\nu}B^{\nu}\|_2=F_{2, I}(\nu),
\end{eqnarray*}
for $A,B $ positive definite matrices and $\nu\in [1/4, 3/4]$.

Setting
$$F_{2, I}(\nu)=\frac 12 \|A^{\nu}B^{1-\nu}+A^{1-\nu}B^{\nu}\|_2\,.$$
 and combining the last inequality with Theorem \ref{conde}, we get the following statement: if $\nu\in [1/4, 3/4]$ and $\alpha\in [1/2, \infty)$, then
\begin{eqnarray*}
\frac 12 \|A^{\nu}B^{1-\nu}+B^{\nu}A^{1-\nu}\|_2&\leq& F_{2, I}(\nu)\leq F_{2, I}(\nu) + 2 \lambda_{\min}\left(\frac{F_{2,I}(\frac 14)+F_{2,I}(\frac 12)}{2}-F_{2,I}\left(\frac 38\right)\right) \nonumber \\ &\leq&(4r_0-1)F_{2,I}\left(\frac12\right)+ 2(1-2r_0)G_{2,I}(\alpha)\leq G_{2,I}(\alpha),
\end{eqnarray*}
where $r_0=\min\{\nu, 1-\nu\},$ $\lambda_{\min}=\min\{2-4r_0, 4r_0-1\}$ and $G_{2,I}(\alpha)=|| (1-\alpha)A^{\frac 12}B^{\frac 12}+\alpha \left(\frac{A+B}{2}\right)||_2.$ In particular, if $\alpha \in [1/2, 1]$ we obtain
\begin{eqnarray*}
\frac 12 \|A^{\nu}B^{1-\nu}+B^{\nu}A^{1-\nu}\|_2&\leq& F_{2, I}(\nu)\leq F_{2, I}(\nu) + 2 \lambda_{\min}\left(\frac{F_{2,I}(\frac 14)+F_{2,I}(\frac 12)}{2}-F_{2,I}\left(\frac 38\right)\right) \nonumber \\ &\leq&(4r_0-1)F_{2,I}\left(\frac12\right)+ 2(1-2r_0)G_{2,I}(\alpha)\leq G_{2,I}(\alpha)\nonumber\\
&\leq& G_{2,I}(1)=\left\|\frac{A+B}{2}\right\|_2.
\end{eqnarray*}
\end{remark}

\textbf{On Zou's questions
}:
In \cite{Z}, the author presented a matrix inequality related to Heinz and Heron means. More precisely, he obtained a matrix version of inequality \eqref{bhatia} for the Schatten norm. If $\nu \in [0,1], A,B \in \mathcal{M}_n^+$, then it holds
\begin{eqnarray}\label{zou}
\frac 12 \|A^{\nu}XB^{1-\nu}+B^{\nu}XA^{1-\nu}\|_2\leq \left\| (1-\alpha(\nu))A^{\frac 12}XB^{\frac 12}+\alpha(\nu) \left(\frac{AX+XB}{2}\right)\right\|_2,\end{eqnarray}
where $\alpha(\nu)=1-4(\nu-\nu^2).$ In that paper, Zou proposed different conjecture or open questions related to inequality \eqref{zou}.

An inequality weaker than \eqref{zou} is
\begin{eqnarray*}
\frac 12 |||A^{\nu}XB^{1-\nu}+B^{\nu}XA^{1-\nu}||| \leq(1-\alpha(\nu))|||A^{\frac 12}XB^{\frac 12}|||+\alpha(\nu) \left|\left|\left|\frac{AX+XB}{2}\right|\right|\right|,
\end{eqnarray*}
with $A, B \in {\mathcal M}_n^+, X \in \mathcal{M}_n$ and for any unitarily invariant norm $|||.|||.$ Zou conjectured that this inequality is true.

Another possible comparison between the Heinz and Heron means is the following:
\begin{eqnarray*}
H_{\nu}(a,b)\leq F_{r_0}(a,b),
\end{eqnarray*}
where $\nu \in [0,1]$ and $r_0=\min\{\nu, 1-\nu\}$. Zou \cite{Z} posed the following question: Is it true that
$$F(\nu)\leq G(1-2r_0)?$$

According to Zou, to answer this question we have to decide whether the function
$$
f(x)=\frac{\cosh(\beta x)}{1-\beta+\beta \cosh (x)}, \qquad 0\leq \beta \leq 1,
$$
is positive definite.
Here, we give a negative answer to this question. For this we choose,
$ \nu=0.42\;\; {\rm{and}}\;\;x_1=1.7006,\; x_2=0\;\; x_3=0.8047 \;\;\;$
then the matrix
$$ \left(\frac{\cosh((1-2\nu)(x_i-x_j))}{2\nu+(1-2\nu)\cosh(x_i-x_j)}\right)=\left(
\begin{array}{ccc}
1 & 0.8023 & 0.9454 \\
0.8023 & 1 & 0.9560 \\
0.9454 & 0.9560 & 1 \\
\end{array}
\right).$$ The determinant of the above matrix turns out to be $-0.0012.$

\section{Refinements of the Cauchy--Schwarz inequality for matrices}

For $A, B \in \mathbb{B}(\mathscr{H})_+$ and any real number $r>0,$ the  inequality
\begin{eqnarray}
\label{Cauchy}||| \: |A^*B|^r\:|||^2\leq ||| (AA^*)^r||| \cdot ||| (BB^*)^r |||,
\end{eqnarray}
is called the operator Cauchy--Schwarz inequality. Let $A, B $ be positive definite matrices and  $X\in \mathcal{M}_n.$ Then, for every positive real number $r$, we consider the function
$$
\phi(t)=|||\: |A^t X B^{1-t}|^r\:|||\cdot|||\: |A^{1-t} X B^t|^r\:|||
$$
which is convex on $[0,1]$ and attains its minimum at $t=\frac 12.$ As a consequence of this last fact, Hiai and Zhan \cite{HX} obtained the following inequality
\begin{eqnarray}\label{hiaizhan}
||| \: |A^{1/2}XB^{1/2}|^r\:|||^2\leq \phi(t) \leq ||| \:|AX|\:^r||| \cdot ||| \:|XB|\:^r |||,
\end{eqnarray}
which is a refinement of \eqref{Cauchy}.

In this section, we utilize the convexity of $\phi(t)$ and Theorem \ref{Mi} to obtain a refinement of the second inequality in \eqref{hiaizhan}.

\begin{theorem}
Let $A, B \in\mathcal{M}_n^{+}$ and  $X\in \mathcal{M}_n.$ Then for $t \in [0,1]$,
\begin{eqnarray*}
\phi(t)\leq \phi(t)+ \lambda_{\min}\left(\frac{\phi(1/2)+\phi(0)}{2}-\phi\left(\frac{1}{4}\right)\right) \leq(1-2t_0)\phi(0)+ 2t_0\phi(1/2),\nonumber \
\end{eqnarray*}
where $t_0=\min\{t, 1-t\}$ and $\lambda_{\min}=\min\{1-2t_0, 2t_0\}$.
\end{theorem}

\begin{proof}
The proof is a consequence of Theorem \ref{Mi}.
\end{proof}

\end{document}